\documentclass[12pt]{article}
\usepackage[utf8]{inputenc}
\usepackage[T1]{fontenc}
\usepackage{array}
\usepackage{amsmath, amsthm}
\usepackage{amssymb}

\newtheorem{theorem}{Theorem}[section]
\newtheorem{proposition}[theorem]{Proposition}
\newtheorem{corollary}[theorem]{Corollary}
\newtheorem{conjecture}[theorem]{Conjecture}
\newtheorem{lemma}[theorem]{Lemma}

\newtheorem{definition}[theorem]{Definition}

\newtheorem{example}[theorem]{Example}
\newtheorem{remark}[theorem]{Remark}

\DeclareMathOperator{\overlap}{overlap}
\DeclareMathOperator{\RunBorder}{R} 
\DeclareMathOperator{\minBorder}{minRun} 
\DeclareMathOperator{\reduce}{rdc}

\DeclareMathOperator{\Alpha}{A}
\DeclareMathOperator{\Factor}{Fac}
\DeclareMathOperator{\PowFactor}{PowFac} 

\DeclareMathOperator{\PL}{PL} 
\DeclareMathOperator{\maxPL}{maxPL}
\DeclareMathOperator{\Prefix}{Pref}
\DeclareMathOperator{\Suffix}{Suf}
\DeclareMathOperator{\Pal}{Pal}
\DeclareMathOperator{\rpo}{rpo} 
\DeclareMathOperator{\rpoDom}{D} 
\DeclareMathOperator{\StdPal}{SP} 

\DeclareMathOperator{\CnStdPal}{CSP} 
\DeclareMathOperator{\maxCSP}{maxCSP} 

\DeclareMathOperator{\PalFactrz}{PF}
\DeclareMathOperator{\PPL}{PPL} 
\DeclareMathOperator{\StdPalFactrz}{SPF}
\DeclareMathOperator{\mirror}{mirror}
\DeclareMathOperator{\factrz}{factrz}

\begin{document}

\title{Palindromic Length and Reduction of Powers}

\author{Josef Rukavicka\thanks{Department of Mathematics,
Faculty of Nuclear Sciences and Physical Engineering, Czech Technical University in Prague
(josef.rukavicka@seznam.cz).}}

\date{\small{March 26, 2021}}

\maketitle

\begin{abstract}
Given a nonempty finite word $v$, let $\PL(v)$ be the palindromic length of $v$; it means the minimal number of palindromes whose concatenation is equal to $v$. Let $v^R$ denote the reversal of $v$. 
Given a finite or infinite word $y$, let $\Factor(y)$ denote the set of all finite factors of $y$ and let $\maxPL(y)=\max\{\PL(t)\mid t\in\Factor(y)\}$.

Let $x$ be an infinite non-ultimately periodic word with $\maxPL(x)=k<\infty$ and let $u\in\Factor(x)$ be a primitive nonempty factor such that $u^5$ is recurrent in $x$. Let $\Psi(x,u)=\{t\in\Factor(x)\mid u,u^R\not\in\Factor(t)\}\mbox{.}$

We construct an infinite non-ultimately periodic word $\overline x$ such that $u^5, (u^R)^5\not\in\Factor(\overline x)$, $\Psi(x,u)\subseteq\Factor(\overline x)$, and $\maxPL(\overline x)\leq 3k^3$.

Less formally said, we show how to reduce the powers of $u$ and $u^R$ in $x$ in such a way that the palindromic length remains bounded.
\end{abstract}

\section{Introduction}
Recall that if $v=a_1a_2\cdots a_n$ is a word of length $n$, where $a_i$ are letters and $i\in \{1,2,\dots,n\}$, then the word $v^R=a_na_{n-1}\cdots a_1$ is called the \emph{reversal} of $v$. We have that $v$ is a \emph{palindrome} if $v=v^R$. The \emph{palindromic length} $\PL(v)$ of the word $v$ is equal to the minimal number $k$ such that $v=v_1v_2\dots v_k$ and $v_j$ are nonempty palindromes, where $j\in \{1,2,\dots,k\}$.
In 2013, Frid, Puzynina, Zamboni conjectured that \cite{FrPuZa}:
\begin{conjecture}
\label{tue8eiru883}
If $w$ is an infinite word and $P$ is an integer such that $\PL(u)\leq P$ for every factor $u$ of $w$ then $w$ is ultimately periodic. 
\end{conjecture}
Conjecture \ref{tue8eiru883} has been solved for some classes of words. Most notably it is known that if $w$ is an infinite $k$-power free word for some positive integer $k$, then the conjecture holds for $w$ \cite{FrPuZa}. Also, for example, the conjecture has been confirmed for Sturmian words \cite{FRID2018202}. However, in general, the conjecture remains open.

In \cite{10.1007/978-3-319-66396-8_19}, another version of Conjecture \ref{tue8eiru883} was presented:
\begin{conjecture}
\label{di88ejdui33}
Every non-ultimately periodic infinite word has prefixes of arbitrarily high palindromic length.
\end{conjecture}

In \cite{10.1007/978-3-319-66396-8_19}, it was proved that Conjecture \ref{tue8eiru883} and Conjecture \ref{di88ejdui33} are equivalent.

There are quite many papers dealing with palindromic length.  In \cite{AMBROZ201974}, the authors study the palindromic length of factors of fixed points of primitive morphisms. In \cite{10.1007/978-3-030-62536-8_14}, the palindromic length of factors with many periodic palindromes is investigated. 

In addition, algorithms for computing the palindromic length were researched \cite{borozdin_et_al:LIPIcs:2017:7338}, \cite{FICI201441}, \cite{RuSh15}. In \cite{RuSh15}, the authors present a linear time online algorithm for computing the palindromic length.

Given an infinite word $w$, let $\PPL_w(n)$ be the palindromic length of the prefix of $w$ with the length $n$. In 2019, Frid conjectured that \cite{10.1007/978-3-030-24886-4_17}:
\begin{conjecture}
\label{fur7ieueifu}
If an infinite word $w$ is $k$-power free for some positive integer $k$, then $\limsup_{n\rightarrow \infty}\frac{\PPL_w(n)}{\ln{n}}>0$.
\end{conjecture}

Let $\mathbb{N}_1$ denote the set of all positive integers and let $\mathbb{R}$ denote the set of all real numbers. It is quite easy to see that if $\varphi(n):\mathbb{N}_1\rightarrow\mathbb{R}$ is a function with $\lim_{n\rightarrow\infty}\varphi(n)=\infty$ and $\varphi(n)<\varphi(n+1)$ for all $n\in\mathbb{N}_1$ (so that the inverse function $\varphi^{-1}$ exists), then there is an infinite non-ultimately periodic word $w$, such that $\limsup_{n\rightarrow\infty}\frac{\PPL_w(n)}{\varphi(n)}\leq 1$. 
To see this, consider the alphabet $\Alpha=\{0,1\}$. Let $w=0^{\lceil\varphi^{-1}(2)\rceil}10^{\lceil\varphi^{-1}(4)\rceil}10^{\lceil\varphi^{-1}(6)\rceil}\cdots$. It is straightforward to prove that $\limsup_{n\rightarrow\infty}\frac{\PPL_w(n)}{\varphi(n)}\leq 1$. It means that the palindromic length of prefixes can grow arbitrarily slow. Note that for every $k\in\mathbb{N}_1$, we have that $0^k$ is a factor of $w$. Thus there is no $k$ such that $w$ is $k$-power free.

Based on this observation and Conjecture \ref{fur7ieueifu}, we could say that the power factors in $w$ allow us to restrict the growth rate of palindromic length of prefixes $\PPL_w(n)$. In the current paper we investigate the relation between the palindromic length and the presence of power factors in an infinite words.  More specifically, we reduce the powers of a given recurrent factor in an infinite word in such a way that the palindromic length does not ``significantly'' change. The basic idea of our powers reduction is as follows: For a given factor $u$ of an infinite word $w$, we replace every factor $u^j$ with $j\in\mathbb{N}_1$ by a factor $u^{\phi(j)}$, where $\phi:\mathbb{N}_1\rightarrow \mathbb{N}_1$ is a function such that $\phi(u)<h$ for some constant $h\in\mathbb{N}_1$. 

One of the motivations for this investigation is this idea: Given an infinite non-ultimately periodic word $x$ with $\maxPL(x)<\infty$, if we could construct an infinite $k$-power free word $\overline x$ with $\maxPL(\overline x)<\infty$ from $x$ by reducing the power factors, we could prove Conjecture \ref{tue8eiru883}.

Given a finite or infinite word $y$, let $\Factor(y)$ denote the set of all finite factors of $y$ and let $\maxPL(y)=\max\{\PL(t)\mid t\in\Factor(y)\}$.
Given an infinite word $x$ and a finite nonempty word $u$, 
let $\Psi(x,u)=\{t\in\Factor(x)\mid u,u^R\not\in\Factor(t)\}\mbox{.}$ 

The main result of the current article is the following theorem.
\begin{theorem}
\label{d7uehgmn}
If $x$ is an infinite non-ultimately periodic word, $u\in\Factor(x)$, $\vert u\vert\geq 1$, $u$ is primitive, $u^5$ is recurrent in $w$, $k\in\mathbb{N}_1$, and $\maxPL(x)\leq k$ then 
there is an infinite non-ultimately periodic word $\overline x$ such that $u^5, (u^R)^5\not\in\Factor(\overline x)$, $\Psi(x,u)\subseteq \Factor(\overline x)$, and $\maxPL(\overline x)\leq 3k^3$.
\end{theorem}
\begin{remark}
Theorem \ref{d7uehgmn} considers infinite non-ultimately periodic words with a bounded palindromic length. Note that according to Conjecture \ref{tue8eiru883}, such words do not exist.
\end{remark}

Let $x$ be an infinite non-ultimately periodic word with $\maxPL(x)<\infty$.

In \cite{BucMichGreedy2018}, it was shown that $x$ has an infinite suffix $\widetilde x$ such that $\widetilde x$ has infinitely many periodic prefixes. It follows that $x$ has only finitely many non-recurrent factors. This result justifies that we consider that $u$ is recurrent.

In \cite{FrPuZa}, it was shown that $x$ does not satisfy the so-called $(k,l)$-condition. It follows that the set $\{u\in\Factor(x)\mid u\mbox{ is primitive and }u^k\in\Factor(x)\}$ is infinite for every $k\in\mathbb{N}_1$. Hence an iterative application of Theorem \ref{d7uehgmn} to construct a $k$-power free non-ultimately periodic word would require an infinite number of iterations. In consequence the constructed word would have an unbounded palindromic length of its factors.

\section{Preliminaries}
Let $\Alpha$ denote a finite alphabet. Given $n\in\mathbb{N}_1$, let $\Alpha^n$ denote the set of all finite words of length $n$ over the alphabet $\Alpha$, let $\epsilon$ denote the empty word, let $\Alpha^+=\bigcup_{i\geq 1}\Alpha^i$, and let $\Alpha^*=\Alpha^+\cup\{\epsilon\}$. Let $\Alpha^{\infty}$ denote the set of all infinite words over the alphabet $\Alpha$; i.e. $\Alpha^{\infty}=\{a_1a_2\dots\mid a_i\in\Alpha\mbox{ and }i\in\mathbb{N}_1\}$.

A word $x\in\Alpha^{\infty}$ is called \emph{ultimately periodic}, if there are $x_1,x_2\in\Alpha^+$ such that $x=x_1x_2^{\infty}$. If there are no such $x_1,x_2$ then $x$ is called \emph{non-ultimately periodic}.

Given a word $v\in\Alpha^*\cup\Alpha^{\infty}$, let $\Factor(v)$ denote the set of all finite factors of $v$ including the empty word. If $v$ is finite then $v\in\Factor(v)$. 

Given a word $v\in\Alpha^*$, let $\Prefix(v)$ and $\Suffix(v)$ denote the set of all prefixes and suffixes of $v$, respectively. We have that $\{\epsilon,v\}\subseteq \Prefix(v)\cap\Suffix(v)$.

Given a word $v\in\Alpha^{\infty}$, let $\Prefix(w)$ denote the set of all finite prefixes of $v$ including the empty word.

Let $\mathbb{N}_0$ denote the set of all nonnegative integers. Let $\mathbb{Q}_1$ denote the set of all rational numbers bigger or equal to $1$. 

Let $q\in\mathbb{Q}_1$. A word $v$ is a $q$-power of the word $u$ if $v=u^{\lfloor q\rfloor}\overline u$, where $\overline u\in\Prefix(u)\setminus\{u\}$ and $\vert v\vert=q\vert u\vert$. We write that $v=u^q$.

A word $v\in\Alpha^+$ is called \emph{primitive} if there are no $u\in\Alpha^+$ and $q\geq 2$ such that $v=u^q$.

Let $\Pal\subseteq\Alpha^{+}$ be the set of all palindromes.

Given $w\in\Alpha^{\infty}$ and $u\in\Factor(w)\setminus\{\epsilon\}$, we say that $u$ is \emph{recurrent} in $w$ if $\vert \{t\in\Prefix(w)\mid u\in\Suffix(t)\}\vert =\infty$.  It means that $u$ has infinitely many occurrences in $w$.

Let $w\in\Alpha^{\infty}$ and $w=a_1a_2\dots$, where $a_i\in\Alpha$ and $i\in\mathbb{N}_1$. We denote  by $w[i,j]$ and $w[i]$ the factors $a_ia_{i+1}\dots a_j$ and $a_i$, respectively, where $i\leq j\in\mathbb{N}_1$. 

\section{Non-ultimately periodic words}

Given $u\in\Alpha^+$, let $\PowFactor(u)=\Factor(u^{\infty})\cup\Factor((u^R)^{\infty})$.
Fix $\gamma\in\mathbb{N}_1$ with $\gamma\geq 3$. Given $w\in\Alpha^{\infty}$, let \[\begin{split}\Pi_{\gamma}(w)=\{u\in\Factor(w)\cap\Alpha^+\mid 
\Prefix(w)\cap\PowFactor(u)\cap \Alpha^{\gamma\vert u\vert}=\emptyset\mbox{ and }\\ u\mbox{ is primitive and there is }v\in\PowFactor(u)\cap\Alpha^{\vert u\vert}\\ \mbox{ such that $v^{\gamma}$ is recurrent in }w\}\mbox{.}\end{split}\]
\begin{remark}
The condition $\Prefix(w)\cap\PowFactor(u)\cap \Alpha^{\gamma\vert u\vert}=\emptyset$ guarantees the power $v^{\gamma}$ is not a prefix of $w$. This is just for our conveniance, when definining the factorization of $w$ in next sections. Since $w$ non-ultimately periodic, there are infinitely many suffixes of $w$ without the prefix $v^{\gamma}$. We apply this obervation in the proof of Theorem \ref{d7uehgmn}.
\end{remark}
Let $w\in\Alpha^{\infty}$ be non-ultimately periodic and let $u\in\Pi_{\gamma}(w)$.

Let \[\begin{split}\RunBorder_{w,u}=\{(i,j)\mid i<j\in\mathbb{N}_1\mbox{ and }j-i+1\geq (\gamma-2)\vert u\vert \mbox{ and }\vert u\vert+1<i \mbox{ and }\\ w[i-\vert u\vert,j+\vert u\vert]\in\PowFactor(u) \mbox{ and }\\
w[i-\vert u\vert,j+\vert u\vert+1]\not\in\PowFactor(u) \mbox{ and }\\ w[i-\vert u\vert-1,j+\vert u\vert]\not\in\PowFactor(u)
\}\mbox{.}\end{split}\]
We call an element $(i,j)\in\RunBorder_{w,u}$ a $u$-\emph{run} or simply a \emph{run}. The term ``run'' has been used also in \cite{FrPuZa}. Our definition of a run is slightly different. Our definition guarantees that runs do not overlap with each other, as shown in the next lemma. In addition, note that if $(i,j)\in\RunBorder_{w,u}$ is a $u$-run, it does not imply that $u=w[i,i+\vert u\vert -1]$.

\begin{lemma}
\label{ryhx6x6c6}
Suppose that  $(i_1,j_1), (i_2,j_2)\in\RunBorder_{w,u}$. We have that
\begin{itemize}
\item If $i_1<i_2$ then $j_1+\vert u\vert+1 <i_2$.
\item If $i_1>i_2$ then $j_2+\vert u\vert+1 <i_1$.
\item If $i_1=i_2$ then $j_1=j_2$.
\item If $j_1=j_2$ then $i_1=i_2$.
\end{itemize}
\end{lemma}
\begin{proof}
Suppose that $i_1<i_2$.
\begin{itemize}
\item
Suppose that $j_1\geq j_2$. Then $w[i_2-\vert u\vert-1,j_2+\vert u\vert]\in\PowFactor(u)$ since $w[i_1-\vert u\vert, j_1+\vert u\vert]\in\PowFactor(u)$. This contradicts that  $(i_2,j_2)\in\RunBorder_{w,u}$. Thus $j_1<j_2$.
\item
Suppose that $j_1<j_2$ and $i_2\leq j_1+\vert u\vert+1$. Let $t_1=w[i_1-\vert u\vert, j_1+\vert u\vert]$, let $t_2=w[i_2-\vert u\vert, j_2+\vert u\vert]$, and let $t=w[i_1-\vert u\vert, j_2+\vert u\vert]$. Then $t_1$ and $t_2$ are periodic words with the period $\vert u\vert$ and $t$ is such that $t_1\in\Prefix(t)$, $t_2\in\Suffix(t)$, and $\vert t\vert\leq \vert t_1\vert+\vert t_2\vert-\vert u\vert$, since $i_2\leq j_1+\vert u\vert+1$. To clarify the inequality the border case $i_2= j_1+\vert u\vert+1$ is depicted in Table \ref{Fig_e78suidf}. It follows that $t$ is periodic with the period $\vert u\vert$ and consequently $w[i_1-\vert u\vert, j_2+\vert u\vert ]\in\PowFactor(u)$. This is a contradiction to that $(i_1,j_1), (i_2,j_2)\in\RunBorder_{w,u}$, since $(i_1,j_1)\in\RunBorder_{w,u}$ implies that \[w[i_1-\vert u\vert, j_1+\vert u\vert +1]\not \in\PowFactor(u)\] and $(i_2,j_2)\in\RunBorder_{w,u}$ implies that \[w[i_2-\vert u\vert-1, j_2+\vert u\vert ]\not\in \PowFactor(u)\mbox{.}\]

We conclude that $j_1+\vert u\vert+1<i_2$.
\begin{table}[ht]
\centering
\begin{tabular}{|c|c|l|l|c|l|c|l|}
\hline
\multicolumn{4}{|c|}{$t$}        \\ \hline
\multicolumn{2}{|c|}{$t_1$} &   \multicolumn{2}{c|}{ }       \\ \hline
$w[i_1-\vert u\vert,j_1]$ & $w[j_1+1, j_1+\vert u\vert]$ & $w[j_1+\vert u\vert+1]$   &       \\ \hline
  & $w[i_2-\vert u\vert, i_2-1]$ & $w[i_2]$ & $w[i_2+1, j_2+\vert u\vert]$          \\ \hline
 &   \multicolumn{3}{c|}{$t_2$ }       \\ \hline
\end{tabular}
\caption{Case for $j_1<j_2$ and $i_2= j_1+\vert u\vert+1$.}
\label{Fig_e78suidf}
\end{table}

\end{itemize}

The case $i_1>i_2$ is analogous. Also it is straightforward to see that $i_1=i_2$ if and only if $j_1=j_2$.
This completes the proof.
\end{proof}


To simplify the presentation of our result we introduce a $\mirror$ function.
Given $i_1\leq i_2\in\mathbb{N}_1$ and $j, j_1, j_2\in\{i_1,i_1+1, \dots, i_2\}$ with $j_1\leq j_2$, let
\begin{itemize}
\item $\mirror(i_1,j,i_2)=\overline j\in\{i_1,i_1+1, \dots, i_2\}$ such that $j-i_1=i_2-\overline j$.
\item $\mirror(i_1,j_1,j_2,i_2)=(m_1,m_2)$, where \[m_1=\mirror(i_1,j_2,i_2)\quad\mbox{ and }\quad m_2=\mirror(i_1,j_1,i_2)\mbox{.}\]
\end{itemize}
\begin{example}
If $i_1=3, j_1=5, j_2=9, i_2=10$ then $m_1=\mirror(3,9,10)=4$, $m_2=\mirror(3,5,10)=8$, and thus $\mirror(3,5,9,10)=(4,8)$.
\end{example}

The next proposition shows that a mirror image of a run in a palindrome is also a run on condition that the border of a palindrome is sufficiently far from the border of the run.
\begin{proposition}
\label{vnm87sm26g}
If $i_1,j_1,j_2,i_2\in\mathbb{N}_1$, $i_1+\vert u\vert< j_1< j_2$, $j_2+\vert u\vert< i_2$, $w[i_1,i_2]\in\Pal$, $(j_1,j_2)\in\RunBorder_{w,u}$, and  $(m_1,m_2)=\mirror(i_1,j_1,j_2,i_2)$ 
 then $(m_1,m_2)\in\RunBorder_{w,u}$.
\end{proposition}
\begin{proof}
From $w[i_1,i_2]\in\Pal$, $i_1+\vert u\vert< j_1< j_2$, and $j_2+\vert u\vert< i_2$, it follows that \begin{equation}\label{rbnch76hdy}w[j_1-\vert u\vert -1,j_2+\vert u\vert +1]=(w[m_1-\vert u\vert-1,m_2+\vert u\vert+1])^R\mbox{.}\end{equation}
Since $(j_1,j_2)\in\RunBorder_{w,u}$ we have that \begin{itemize}\item $w[j_1-\vert u\vert ,j_2+\vert u\vert]\in\PowFactor(u)$, \item $w[j_1-\vert u\vert-1,j_2+\vert u\vert]\not\in\PowFactor(u)$, and \item if $w[j_1-\vert u\vert,j_2+\vert u\vert+1]\not\in\PowFactor(u)$.\end{itemize} From (\ref{rbnch76hdy}) it follows that \begin{itemize}\item $w[m_1-\vert u\vert,m_2+\vert u\vert]\in\PowFactor(u)$, \item $w[m_1-\vert u\vert-1,m_2+\vert u\vert]\not\in\PowFactor(u)$, and \item $w[m_1-\vert u\vert,m_2+\vert u\vert+1]\not\in\PowFactor(u)$.\end{itemize} Consequently $(m_1,m_2)\in\RunBorder_{w,u}$. This completes the proof.
\end{proof}

\section{Factorization and the reduced word}
We define an order on the set $\RunBorder_{w,u}$ as follows: $(i_1,j_1)<(i_2,j_2)$ if and only if $i_1<i_2$, where $(i_1,j_1)\not= (i_2,j_2)\in\RunBorder_{w,u}$. Lemma \ref{ryhx6x6c6} implies that $\RunBorder_{w,u}$ is totally ordered and since $i$ is a positive integer for every $(i,j)\in\RunBorder_{w,u}$, we have also that $\RunBorder_{w,u}$ is well ordered. It means that the function $\min\{S\}$ is well defined for every $S\subseteq \RunBorder_{w,u}$.

Let $\minBorder_{w,u}: \mathbb{N}_1\rightarrow \RunBorder_{w,u}$ be defined as follows:
\[\minBorder_{w,u}(1)=\min\{\RunBorder_{w,u}\}\mbox{ }\]
and for every $k\in\mathbb{N}_1$ we define that \[\minBorder_{w,u}(k+1)=\min\{\RunBorder_{w,u}\setminus \bigcup_{i=1}^k\{\minBorder_{w,u}(i)\}\}\mbox{.}\]

We define a special factorization of the word $w$, that will allow us to construct an infinite power free word based on $w$ with ``reduced'' powers of $u$.
\begin{definition}
Let $\factrz(w,u)=(\overrightarrow{w}, \overrightarrow{z}, \overrightarrow{d})$ be defined as follows:
\begin{itemize}
\item $k\in\mathbb{N}_1$,
\item $j_0=0$, $(i_k,j_k)=\minBorder_{w,u}(k)$,
\item $w_k=w[j_{k-1}+1,i_{k}-1]$,
\item $z_k\in\PowFactor(u)\cap\Alpha^{\vert u\vert}$ and $d_k\in\mathbb{Q}_1$ are such that $w[i_k,j_k]=z_k^{d_k}$, 
\item $\overrightarrow{w}=(w_1,w_2,\dots)$, $\overrightarrow{z}=(z_1,z_2,\dots)$, and $\overrightarrow{d}=(d_1,d_2,\dots)$.
\end{itemize}
\end{definition}
\begin{remark}It is clear that $\factrz(w,u)$ exists and is uniquely determined.  Obviously we have that $w=w_1z_1^{d_1}w_2z_2^{d_2}\dots\mbox{.}$
\end{remark}

We define a set of functions, that we will use to ``reduce'' the exponents $d_k$ of factors $z_k$.
Fix $h\in\mathbb{N}_1$ with $h\geq \gamma$. Let \[\begin{split}\Phi_h=\{\phi:\mathbb{Q}_1\rightarrow\mathbb{Q}_1\mid \gamma-2\leq \phi(q)< h\mbox{ and }q-\phi(q)\in\mathbb{N}_0 \}\mbox{.}\end{split}\]

The next lemma, less formally said, shows that the reduction by $\phi\in\Phi_h$ preserves the reverse relation.
\begin{lemma}
\label{fyue7dyiiu7d}
If $\phi\in\Phi_h$, $q\in\mathbb{Q}_1$, $t,v\in\Alpha^{+}$, $\vert t\vert =\vert v\vert$, and $v^q=(t^q)^R$, then $v^{\phi(q)}=(t^{\phi(q)})^R$.
\end{lemma}
\begin{proof}
Let $\vert t\vert=k$.
Then $v=v_1v_2\dots v_k$ and $t=t_1t_2\dots t_k$, where $v_i,t_i\in\Alpha$ and $i\in\{1,2,\dots,k\}$. Let $p=q\pmod{1}$. 


We have that $p\in\{0,\frac{1}{k}, \frac{2}{k}, \dots, \frac{k-1}{k}\}$. Let $j\in\{0,1,2,\dots,k-1\}$ be such that $p=\frac{j}{k}$.
Then $t^q=t^{\lfloor q\rfloor} t_1t_2\dots t_j$ and $v=v^{\lfloor q\rfloor}v_1v_2\dots v_j$.
Since $q-\phi(q)\in\mathbb{N}_0$, we have that $t^{\phi(q)}=t^{\lfloor \phi(q)\rfloor} t_1t_2\dots t_j$ and $v^{\phi(i)}=v^{\lfloor \phi(q)\rfloor}v_1v_2\dots v_j$.
Because $1\leq \gamma-2\leq \phi(q)$, it follows that if $v^q=(t^q)^R$ then $v^{\phi(q)}=(t^{\phi(q)})^R$.
This completes the proof.
\end{proof}

Fix $(\overrightarrow{w}, \overrightarrow{z}, \overrightarrow{d})=\factrz(w,u)$.
\begin{definition}
Given $\phi\in\Phi_h$, 
let \[\reduce_{w,u,\phi}=w_1z_1^{\phi(d_1)}w_2z_2^{\phi(d_2)}\cdots \mbox{.}\]
\end{definition}
We call the word $\reduce_{w,u,\phi}\in\Alpha^{\infty}$ the \emph{reduced word} of $w$. The basic property of the reduced word $\reduce_{w,u,\phi}$ is that powers of $u$ in $\reduce_{w,u,\phi}$ are bounded by $h+2$.
\begin{lemma}
\label{tufdkd3455}
If $\phi\in\Phi_h$ and $t\in\PowFactor(u)\cap\Alpha^{\vert u\vert}$ then $t^{h+2}\not\in\Factor(\reduce_{w,u,\phi})$. 
\end{lemma}
\begin{proof}
Realize that $\phi(q)< h$ for all $q\in\mathbb{Q}_1$ and if $(i,j)\in\RunBorder_{w,u}$, then \[w[i-\vert u\vert-1, j+\vert u\vert]\not\in\PowFactor(u)\quad\mbox{ and }\quad w[i-\vert u\vert, j+\vert u\vert+1]\not\in\PowFactor(u)\mbox{.}\]
The lemma follows.
\end{proof}

For our convenience when working with the reduced word $\reduce_{w,u,\phi}$ we introduce some more functions.
Given $j\in\mathbb{N}_1$ and $\phi\in\Phi_h$, let
\[\kappa_{w,u}(j)=\sum_{i=1}^j(\vert w_j\vert +d_j\vert z_j\vert)\quad\mbox{ and }\]
\[\overline\kappa_{w,u,\phi}(j)=\sum_{i=1}^j(\vert w_j\vert +\phi(d_j)\vert z_j\vert)\mbox{.}\]
In addition we define $\kappa_{w,u}(0)=0$ and $\overline \kappa_{w,u,\phi}(0)=0$.

We define a set $\widetilde \RunBorder$ to be the set of positions that are covered by runs; formally 
let $\widetilde\RunBorder_{w,u}=\{i\in\mathbb{N}_1\mid \mbox{ there is } (i_1,i_2)\in\RunBorder_{w,u}\mbox{ such that } i_1\leq i\leq i_2\}\mbox{.}$
Let $\rpoDom_{w,u}=\mathbb{N}_1\setminus\widetilde \RunBorder_{w,u}\mbox{.}$ and 
let $\rpo_{w,u,\phi}: \rpoDom_{w,u}\rightarrow \mathbb{N}_1$ be a function defined as follows:
\begin{itemize}
\item Given $i\in\rpoDom_{w,u}$, let $j\in\mathbb{N}_0$ and $f\in\mathbb{N}_1$ be such that $\kappa_{w,u}(j)< i\leq\kappa_{w,u}(j+1)$ and $i=f+\kappa_{w,u}(j)$.
\item We define that $\rpo_{w,u,\phi}(i)=\overline\kappa_{w,u,\phi}(j)+f$.
\end{itemize}

We will need the following simple lemma.
\begin{lemma}
\label{ryg67hr778f}
For every $k\in\mathbb{N}_1$ we have that $\vert w_k\vert> \vert u\vert$.
\end{lemma}
\begin{proof}
For $k=1$ the the lemma follows from the definition of $\RunBorder_{w,u}$.
For $k\geq 2$, the lemma from Lemma \ref{ryhx6x6c6}.

This ends the proof.
\end{proof}

The next lemma says that the function $\rpo_{w,u,\phi}$ forms a ``natural'' bijection between $u$-runs of $w$ and $u$-runs of the reduced word. Realize that $p_j=\kappa_{w,u}(j)$ is the end position of the $j$-th run of $w$. Then we have that $p_j+1\in\rpoDom_{w,u}$ and consequently $n_j=\rpo_{w,u,\phi}(p_j+1)$ is defined. We show that $n_j-1$ is the end position of the $j$-th run in the reduced word.
\begin{proposition}
\label{tu7bdj80kj}
If $\phi\in\Phi_h$, $\overline w=\reduce_{w,u,\phi}$, and $\overrightarrow{f}=(\phi(d_1),\phi(d_2),\dots)$ then 
$(\overrightarrow{w}, \overrightarrow{z}, \overrightarrow{f})=\factrz(\overline w,u)$.
\end{proposition}
\begin{proof}
We have that $w=w_1z_1^{d_1}w_2z_2^{d_2}\cdots$ and that $\overline w=w_1z_1^{\phi(d_1)}w_2z_2^{\phi(d_2)}\cdots$.

The definition of $\RunBorder_{w,u}$ implies that for every $i\in\mathbb{N}_1$ we have that $z_i\in\Suffix(w_i)$, and there is $\widehat z_i\in\Factor(z_iz_i)\cap\Alpha^{\vert z_i\vert}$ such that $\widehat z_i\in\Prefix(w_{i+1})$. Let $a_i,b_i\in\Alpha$ be such that $a_iz_i\in\Suffix(w_i)$ and $\widehat z_ib_i\in\Prefix(w_{i+1})$.
Lemma \ref{ryg67hr778f} implies that such $a_i,b_i$ exist, because $\vert w_i\vert\geq \vert u\vert+1=\vert z_i\vert+1$.

Let $\Omega=\bigcup_{j=\mathbb{N}_1}(\overline \kappa_{w,u,\phi}(j)+\vert w_{j+1}\vert +1,\overline \kappa_{w,u,\phi}(j+1)-1)$.

We have that $z_iz_i^{d_i}\widehat z_i\in\PowFactor(u)$. Since $d_i-\phi(d_i)\in\mathbb{N}_0$, it follows easily that $z_iz_i^{\phi(d_i)}\widehat z_i\in\PowFactor(u)$. Since $a_iz_iz_i^{d_i}\widehat z_i\not\in\PowFactor(u)$ and $z_iz_i^{d_i}\widehat z_ib_i\not\in\PowFactor(u)$, it follows also that $a_iz_iz_i^{\phi(d_i)}\widehat z_i\not\in\PowFactor(u)$ and $z_iz_i^{\phi(d_i)}\widehat z_ib_i\not\in\PowFactor(u)$. Since $\phi(d_i)\geq \gamma-2$ and $d_i-\phi(d_i)\in\mathbb{N}_0$, this implies that $z_i^{\gamma}\in\Factor(\overline w)$ and consequently that $u\in\Pi_{\gamma}(\overline w)$ and $\Omega\subseteq \RunBorder_{\overline w,u}$. 

Suppose that there is $(m_1, m_2)\in\RunBorder_{\overline w,u}\setminus \Omega$.
From the definition of $\RunBorder_{\overline w,u}$ we have that \begin{itemize}\item $\overline w[m_1-\vert u\vert -1, m_2+\vert u\vert]\not\in\PowFactor(u)$, \item $\overline w[m_1-\vert u\vert , m_2+\vert u\vert+1]\not\in\PowFactor(u)$, and \item $\overline w[m_1-\vert u\vert , m_2+\vert u\vert]\in\PowFactor(u)$. \end{itemize}
Lemma \ref{ryhx6x6c6} implies that there is $j\in\mathbb{N}_1$ such that 
\[\overline \kappa_{w,u,\phi}(j)+\vert u\vert +1< m_1\leq m_2\quad\mbox{ and }\quad m_2+\vert u\vert +1\leq \overline \kappa_{w,u,\phi}(j)+\vert w_{j+1}\vert\] and consequently $\overline w[m_1-\vert u\vert-1, m_2+\vert u\vert+1]\in\Factor(w_j)$. This is a contradiction, since by definition the factors $w_j$ do not contain $u$-runs.

It follows that $\RunBorder_{\overline w,u}=\Omega$. This completes the proof.
\end{proof}
From the proof of Proposition \ref{tu7bdj80kj} it follows also that $\rpo_{w,u,\phi}$ is a bijection between positions of $w$ and $\overline w$ that are not covered by $u$-runs.
\begin{corollary}
\label{dyf7ruryre}
If $\phi\in\Phi_h$ and $\overline w=\reduce_{w,u,\phi}$ then $\rpo_{w,u,\phi}: \rpoDom_{w,u}\rightarrow\rpoDom_{\overline w,u}$ is a bijection.
\end{corollary}

For our main result we need that $\reduce_{w,u,\phi}$ is non-ultimately periodic. We show we can select $\phi$ in such a way that this requirement is satisfied.
\begin{theorem}
\label{fyur7t7tj}
There is $\varphi\in\Phi_h$ such that $\reduce_{w,u,\varphi}$ is non-ultimately periodic. 
\end{theorem}
\begin{proof}Clearly $\Phi_h\not=\emptyset$. 
Let $\phi\in\Phi_h$ and let $\overline w=\reduce_{w,u,\phi}$. If $\overline w$ is non-ultimately periodic, then let $\varphi=\phi$ and we are done. 
Suppose that $\overline w$ is ultimately periodic.
Let $p_0,p\in\mathbb{N}_1$ be such that $w[p_0+j+ip]=w[p_0+j]$ for every $i,j\in\mathbb{N}_0$. 

We have that $\overline w=w_1z_1^{\phi(d_1)}w_2z_2^{\phi(d_2)}\dots$. Proposition \ref{tu7bdj80kj} implies that 
$(\overrightarrow{w}, \overrightarrow{z}, \overrightarrow{f})=\factrz(\overline w,u)$, where $\overrightarrow{f}=(\phi(d_1),\phi(d_2),\dots)$.

Let $m_j=\kappa_{\overline w,u}(j)$. If $j\in\mathbb{N}_1$ is such that $m_{j-1}>p_0$ then obviously \[(m_{j-1}+\vert w_j\vert+1+ip, m_j+ip)\in\RunBorder_{\overline w,u}\mbox{,}\] since $(m_{j-1}+\vert w_j\vert+1, m_j)\in\RunBorder_{\overline w,u}$ and $\overline w$ is ultimately periodic. It follows that there are $k<n\in\mathbb{N}_j$ such that 
\[\overline w=w_1z_1^{\phi(d_1)}w_2z_2^{\phi(d_2)}\dots w_{k-1}z_{k-1}^{f_{k-1}}\left(w_{k}z_{k}^{f_{k}}w_{k+1}z_{k+1}^{f_{k+1}}\dots w_{n}z_{n}^{f_{n}}\right)^{\infty}\mbox{.}\]

Since $w$ is non-ultimately periodic, it follows that there is $j\in\mathbb{N}_1$ such that for every $i\in\mathbb{N}_1$ we have that $\phi(d_j)=\phi(d_{j+i(n-k+1)})$ and the sequence $(d_j, d_{j+n-k+1}, d_{j+2(n-k+1)}, d_{j+3(n-k+1)}, \dots)$ is non-ultimately periodic.


It is straightforward to verify that there exists $\varphi\in\Phi_{h}$ such that the sequence \[(\varphi(d_j), \varphi(d_{j+n-k+1}), \varphi(d_{j+2(n-k+1)}), \varphi(d_{j+3(n-k+1)}), \dots)\] is non-ultimately periodic. In consequence $\reduce_{w,u,\varphi}$ is also non-ultimately periodic.

This completes the proof.
\end{proof}

Let us fix $\phi\in\Phi_h$ such that $\reduce_{w,u,\phi}$ is non-ultimately periodic. Theorem \ref{fyur7t7tj} asserts that such $\phi$ exists.

We will need the following elementary Lemma concerning the palindromic length. We omit the proof.
\begin{lemma}
\label{dujf720bxg}
If $t_1,t_2\in\Alpha^+$ then $\PL(t_1t_2)\leq \PL(t_1)+\PL(t_2)$.
\end{lemma}

Given $i_1\leq i_2\in\mathbb{N}_1$, let \[\RunBorder_{w,u}(i_1,i_2)=\{(i,j)\in\RunBorder_{w,u}\mid i_1\leq i\mbox{ and }j\leq i_2\}\mbox{.}\]

By means of the function $\rpo_{w,u,\phi}$ we map the positions of $w$ into the positions of the reduced word $\reduce_{w,u,\phi}$. Thus by mapping two positions $i\leq j$, we get a map of factors $w[i,j]$ on the factors of $\reduce_{w,u,\phi}$. The next proposition shows an upper bound on the palindromic length of the factors of $\reduce_{w,u,\phi}$ as a function of the number of runs in the preimage of these factors of $\reduce_{w,u,\phi}$.
\begin{proposition}
\label{kkm3bvs56}
If $m_1\leq m_2\in\rpoDom_{w,u}$, $g\in\mathbb{N}_1$, $\vert \RunBorder_{w,u}(m_1,m_2)\vert\leq g$, $\overline m_1=\rpo_{w,u,\phi}(m_1)$, and $\overline m_2=\rpo_{w,u,\phi}(m_2)$ then 
\[\PL(\reduce_{w,u,\phi}[\overline m_1,\overline m_2])\leq (g+1) \maxPL(w[m_1,m_2])\mbox{.}\]
\end{proposition}
\begin{proof}
The lemma is obvious for $\vert \RunBorder_{w,u}(m_1,m_2)\vert=0$, because then there is $j\in\mathbb{N}_1$ such that $w[m_1,m_2]\in\Factor(w_j)$ and consequently $w[m_1,m_2]=\reduce_{w,u,\phi}[\overline m_1,\overline m_2]$.

Suppose $\vert \RunBorder_{w,u}(m_1,m_2)\vert\geq 1$.
Since $\vert \RunBorder_{w,u}(m_1,m_2)\vert\leq g$ and $m_1\leq m_2\in\rpoDom_{w,u}$, there are $j,k\in\mathbb{N}_1$ such that 
\begin{itemize}
\item $w[m_1,m_2]=\widehat w_jz_j^{d_j}w_{j+1}z_{j+1}^{d_{j+1}}\dots \widehat w_k$,
\item $\widehat w_j\in\Suffix(w_j)\setminus\{\epsilon\}$, 
\item $\widehat w_k\in\Prefix(w_k)\setminus\{\epsilon\}$, and
\item $k-j\leq g$.
\end{itemize}

Proposition \ref{tu7bdj80kj} implies that $\reduce_{w,u,\phi}[\overline m_1,\overline m_2]=\widehat w_jz_j^{\phi(d_j)}w_{j+1}z_{\phi(j+1)}^{\phi(d_{j+1})}\dots \widehat w_k$. Since $d_i-\phi(d_i)\in\mathbb{N}_0$ we have that $d_i\geq\phi(d_i)$ and $w_iz_i^{\phi(d_i)}\in\Prefix(w_iz_i^{d_i})$. Obviously $\PL(v)\leq \maxPL(w[m_1,m_2])$ for every factor $v\in\Factor(w_iz_i^{\phi(d_i)})$.

The proposition then follows from Lemma \ref{dujf720bxg}. This ends the proof. 
\end{proof}

\section{Standard Palindromes}
Let \[\begin{split}\StdPal_{w,u}=\{(i,j)\mid i\leq j\in\rpoDom_{w,u}\mbox{ and } w[i-1,j+1]\in\Pal\mbox{ and }\\ \{i-1,i,i+1,i+2, \dots, \min\{i+\vert u\vert-1, j\}\}\subseteq \rpoDom_{w,u} \mbox{ and }\\ \{j+1,j,j-1,j-2, \dots,\max\{j-\vert u\vert +1 ,i\}\}\subseteq \rpoDom_{w,u} \mbox{ and } \\ 
\mbox{ for every }m\in\{i,i+1, \dots,j\}\mbox{ we have that }\\
m\in\rpoDom_{w,u}\mbox{ if and only if }\mirror(i,m,j)\in\rpoDom_{w,u}\}
\mbox{.}\end{split}\]
The elements of the set $\StdPal_{w,u}$ are called \emph{standard palindromes} of $w$. 
The key property of a standard palindrome of $w$ is that its image in the reduced word is also a palindrome.
\begin{proposition}
\label{jt77gjri29k}
If $i_1\leq i_2\in\mathbb{N}_1$ and $(i_1,i_2)\in\StdPal_{w,u}$  then 
\[\reduce_{w,u,\phi}[\rpo_{w,u,\phi}(i_1), \rpo_{w,u,\phi}(i_2)]\in\Pal\mbox{.}\]
\end{proposition}
\begin{proof}
Since $(i_1,i_2)\in\StdPal_{w,u}$, we have that $i_1,i_2\in\rpoDom_{w,u}$. Thus let 
$\overline i_1=\rpo_{w,u,\phi}(i_1)$ and let $\overline i_2=\rpo_{w,u,\phi}(i_2)$.

If $\RunBorder_{w,u}(i_1,i_2)=\emptyset$ then it is easy to see that $\reduce_{w,u,\phi}[\overline i_1,\overline i_2]=w[i_1,i_2]$ and consequently $\reduce_{w,u,\phi}[\overline i_1,\overline i_2]\in\Pal$. 

Let us suppose that $\RunBorder_{w,u}(i_1,i_2)\not=\emptyset$.
Then there are $j<k\in\mathbb{N}_1$ such that 
$\widehat w_j\in\Suffix(w_j)\setminus\{\epsilon\}$, 
$\widehat w_k\in\Prefix(w_k)\setminus\{\epsilon\}$, 
\begin{equation}\label{jdud837ehj}w[i_1,i_2]=\widehat w_{j}z_{j}^{d_j}w_{j+1}z_{j+1}^{d_{j+1}}\dots w_{k-1}z_{k-1}^{d_{k-1}}\widehat w_k\mbox{, and }\end{equation}
\begin{equation}\label{bbnc829efg}\reduce_{w,u,\phi}[\overline i_1,\overline i_2]=\widehat w_{j}z_{j}^{\phi(d_j)}w_{j+1}z_{j+1}^{\phi(d_{j+1})}\dots w_{k-1}z_{k-1}^{\phi(d_{k-1})}\widehat w_k\mbox{.}\end{equation}

Let $(m_1,m_2)\in\RunBorder_{w,u}(i_1,i_2)$. From the definition of $\StdPal_{w,u}$ we get that
\begin{itemize}
\item $w[i_1-1,i_2+1]\in\Pal$, 
\item $\{i_1-1, i_1,i_1+1,i_1+2,\dots \dots, i_1+\vert u\vert -1 \}\subseteq\rpoDom_{w,u}$, and
\item $\{i_2+1, i_2,i_2-1,i_2-2, \dots, i_2-\vert u\vert +1\}\subseteq\rpoDom_{w,u}$.
\end{itemize}
Hence Proposition \ref{vnm87sm26g} implies that $\mirror(i_1,m_1,m_2,i_2)\in\RunBorder_{w,u}(i_1,i_2)$.
In consequence it follows from Proposition \ref{tu7bdj80kj} and  (\ref{jdud837ehj})  that 
\begin{itemize}\item $\widehat w_{j}=(\widehat w_k)^R$, 
\item $w_{j+n}=(w_{k-n})^R$ for every $n\in\mathbb{N}_1$ such that $j+n<k$, and
\item $z_{j+n}^{d_{j+n}}=(z_{k-1-n}^{d_{k-1-n}})^R$ for every $n\in\mathbb{N}_0$ such that $j+n<k$.
\end{itemize}
Lemma \ref{fyue7dyiiu7d} implies that $z_{j+n}^{\phi(d_{j+n})}=(z_{k-1-n}^{\phi(d_{k-1-n})})^R$ for every $n\in\mathbb{N}_0$ such that $j+n<k$.

Thus from (\ref{jdud837ehj}) and (\ref{bbnc829efg}) we get that $w[i_1,i_1]\in \Pal$ implies $\reduce_{w,u,\phi}[\overline i_1,\overline i_2]\in\Pal$.
This completes the proof.
\end{proof}

Given $i\leq j\in\mathbb{N}_1$, let 
\[\begin{split}\CnStdPal_{w,u}(i,j)=
\{(m_1,m_2)\in\StdPal_{w,u}\mid i\leq m_1\leq m_2\leq j\mbox{ and } \\ m_1-i=j-m_2 \}\mbox{ } \end{split}\mbox{.}\]
We call the elements of $\CnStdPal_{w,u}(i,j)$ \emph{centered standard palindromes} of $w[i,j]$. 

The next lemma shows under which condition a palindrome $w[i,j]$ contains a centered standard palindrome.
\begin{lemma}
\label{ufj77cbdjh22}
If $i\leq j\in\mathbb{N}_1$, $w[i,j]\in\Pal$, $(m_1,m_2)\in\RunBorder_{w,u}(i,j)$, $i+\vert u\vert<m_1$, $m_2+\vert u\vert<j$, $(m_3,m_4)=\mirror(i,m_1,m_2,j)$, $k=\min\{m_1,m_3\}$, and $\overline k=\mirror(i,k,j)$ then $(k-\vert u\vert,\overline k+\vert u\vert)\in\CnStdPal_{w,u}(i,j)$.
\end{lemma}
\begin{proof}
Proposition \ref{vnm87sm26g} implies that for every \[n\in\{k-\vert u\vert-1, k-\vert u\vert, k-\vert u\vert +1, \dots, \overline k+\vert u\vert+1\}\] we have that $n\in\rpoDom_{w,u}$ if and only if $\mirror(i,n,j)\in\rpoDom_{w,u}$. It follows that $(m_3,m_4)\in\RunBorder_{w,u}$.
Since $(m_1,m_2), (m_3,m_4)\in\RunBorder_{w,u}(i,j)$, Lemma $\ref{ryhx6x6c6}$ implies that 
$\{k-1, k-2, \dots, k-\vert u\vert -1\}\subseteq \rpoDom_{w,u}$, and $\{\overline k +1, \overline k+2, \dots, \overline k+\vert u\vert +1\}\subseteq \rpoDom_{w,u}$.

Thus $(k-\vert u\vert,\overline k+\vert u\vert)\in\StdPal_{w,u}$. Since $k-i=j-\overline k$, we have also that  $(k-\vert u\vert,\overline k+\vert u\vert)\in\CnStdPal_{w,u}(i,j)$. This completes the proof.
\end{proof}

The following lemma proves that if the palindrome $w[i,j]$ contains at least three runs, then it contains also a centered standard palindrome.
\begin{lemma}
\label{uj827xzv3}
If $(i,j)\in\widehat\Upsilon$ then 
$\vert \RunBorder_{w,u}(i,j)\vert\leq 2$.
\end{lemma}
\begin{proof}
Let $(i_1,i_2), (i_3,i_4), (i_5,i_6)\in\RunBorder_{w,u}(i,j)$ with $i_1<i_3<i_5$. 

Let $(m_1,m_2)=\mirror(i,i_3,i_4,j)$. Let $k=\min\{m_1,i_3\}$ and let $\overline k=\mirror(i,k,j)$. Lemma \ref{ryhx6x6c6} implies that $i_2+\vert u\vert +1<i_3$ and hence Lemma \ref{ufj77cbdjh22} implies that $(k-\vert u\vert, \overline k+\vert u\vert)\in\CnStdPal_{w,u}(i,j)$. We proved that if $\vert \RunBorder_{w,u}(i,j)\vert\geq 3$ then $\CnStdPal_{w,u}(i,j)\not=\emptyset$. The lemma follows.
\end{proof}

We define an order on the set $\CnStdPal_{w,u}(i,j)$ as follows: If $(i_1,j_1), (i_2,j_2)\in\CnStdPal_{w,u}(i,j)$ and $(i_1,j_1)\not=(i_2,j_2)$ then $(i_1,j_1)<(i_2,j_2)$ if and only if $j_1-i_1<j_2-i_2$. Thus the function $\max\{\CnStdPal_{w,u}(i,j)\}$ is well defined on condition that $\CnStdPal_{w,u}(i,j)\not=\emptyset$. Let $\maxCSP_{w,u}(i,j)=\max\{\CnStdPal_{w,u}(i,j)\}$.
\begin{remark}
Since the elements of $\CnStdPal_{w,u}(i,j)$ are ``centered'', the inequality $j_1-i_1<j_2-i_2$ is equivalent to the inequalities $j_1<j_2$ and $i_1>i_2$.
\end{remark}

To simplify our next results, we introduce three auxiliary sets and a function $\overlap:\mathbb{N}_1^2\times\mathbb{N}_1^2\rightarrow \{0,1\}$.
Let  \[\begin{split}\widehat \Upsilon_{w,u}=\{(i,j)\mid i\leq j\in\mathbb{N}_1\mbox{ and }w[i,j]\in\Pal\mbox{ and }\\ \CnStdPal_{w,u}(i,j)=\emptyset\}\mbox{.}\end{split}\]

Given $i_1\leq j_1\in\mathbb{N}_1$ and $i_2\leq j_2\in\mathbb{N}_1$, let 
\[
  \overlap((i_1,j_1), (i_2,j_2)) =
  \begin{cases}
    0 & \text{if $i_1\leq i_2\leq j_1$} \\
    0 & \text{if $i_2\leq i_1\leq j_2$} \\
    1 & \text{otherwise.}
  \end{cases}
\]

Let  \[\begin{split}\overline \Upsilon_{w,u}=\{(i,j)\mid i\leq j\in\mathbb{N}_1\mbox{ and there are }\overline i\leq \overline j\in\mathbb{N}_1\mbox{ such that }\\ \overline i\leq i\leq j\leq \overline j\mbox{ and }w[\overline i, \overline j]\in\Pal\mbox{ and }\\ \CnStdPal_{w,u}(\overline i, \overline j)\not=\emptyset\mbox{ and }\overlap((i,j),\maxCSP_{w,u}(\overline i, \overline j))=0\}\mbox{.}\end{split}\]

Let $\Upsilon_{w,u}=\widehat \Upsilon_{w,u}\cup _{w,u}\overline \Upsilon_{w,u}$.

Less formally said, $\overline\Upsilon$ contains positions $(i,j)$ of factors of a palindrome $w[\overline i,\overline j]$ that are between the border of the palindrome $w[\overline i,\overline j]$ and the border of its maximal centered standard palindrome. The next proposition shows that these factors contain at most one run.
\begin{proposition}
\label{ryy8hu8cxvxc}
If $(n_1,n_2)\in\overline\Upsilon_{w,u}$ then  $\vert \RunBorder_{w,u}(n_1,n_2)\vert\leq 1$.
\end{proposition}
\begin{proof}
Let $i\leq j\in\mathbb{N}_1$ be such that $i\leq n_1\leq n_2\leq j$, $w[i,j]\in\Pal$,  $\maxCSP_{w,u}(i,j)\not=\emptyset$, and \begin{equation}\label{gnhd8d787}\overlap((n_1,n_2),\maxCSP_{w,u}(i,j))=0\mbox{.}\end{equation} 
Since $(n_1,n_2)\in\overline\Upsilon_{w,u}$, we know that such $i,j$ exist; however note that $i,j$ are not uniquely determined.

Let $(m_1,m_2)=\maxCSP_{w,u}(i,j)$. It follows from (\ref{gnhd8d787}) that \[i\leq n_1\leq n_2<m_1\quad\mbox{ or }\quad m_2<n_1\leq n_2\leq j\mbox{.}\]
Without loss of generality suppose that $i\leq n_1\leq n_2<m_1$.

Suppose that there are $(i_1,i_2), (i_3,i_4)\in\RunBorder_{w,u}(i,m_1)$ with $i_1<i_3$. Lemma \ref{ryhx6x6c6} implies that $i_2+\vert u\vert +1<i_3$ and hence Lemma \ref{ufj77cbdjh22} implies that \[(i_3-\vert u\vert, \mirror(i,i_3,j)+\vert u\vert)\in\CnStdPal_{w,u}(n_1,n_2)\mbox{.}\] This contradicts to $(m_1,m_2)=\maxCSP_{w,u}(i,j)$.
We conclude that $\vert\RunBorder_{w,u}(i,m_1)\vert\leq 1$.

This completes the proof.
\end{proof}

\section{Palindromic factorization}
Let $n_1\leq n_2\in\mathbb{N}_1$, $k\in\mathbb{N}_1$, and let \[\begin{split}\PalFactrz_{w,k}(n_1,n_2)=\{(m_1,m_2,\dots,m_j)\mid j\leq k\mbox{ and } \\ m_1\leq m_2\leq \dots \leq m_j\in\mathbb{N}_1\mbox{ and }\\ n_1=m_1\mbox{ and }n_2=m_j-1\mbox{ and }\\ w[m_i,m_{i+1}-1]\in\Pal\mbox{ for } i\in\{1,2,\dots,j-1\} \mbox{ and }\\ w[n_1,n_2]=w[m_1,m_2-1]w[m_2,m_3-1]\dots w[m_{j-1},m_j-1]\}\mbox{.}\end{split}\]
We call the elements of $\PalFactrz_{w,k}(n_1,n_2)$ \emph{palindromic factorizations} of the factor $w[n_1,n_2]$.

Let $n_1\leq n_2\in\mathbb{N}_1$, $k\in\mathbb{N}_1$, and
let \[\begin{split}\StdPalFactrz_{w,u,k}(n_1,n_2)=\{(\delta_1,\delta_2,\dots,\delta_g)\mid g\leq k\mbox{ and }\\ \delta_1\leq \delta_2\leq \dots\leq\delta_g\in\mathbb{N}_1\mbox{ and }\\
n_1=\delta_1\mbox{ and }n_2=\delta_g-1\mbox{ and }\\ 
w[n_1,n_2]=w[\delta_1, \delta_2-1]w[\delta_2,\delta_3-1]\cdots w[\delta_{g-1}, \delta_g-1]\mbox{ and }\\
\mbox{ for every }i\in\{1,2,\dots, g-1\}\mbox{ we have that }\delta_i, \delta_{i+1}-1\in\rpoDom_{w,u}\mbox{ and }\\
\mbox {if }(\delta_i,\delta_{i+1}-1)\not\in\StdPal_{w,u}\mbox{ then  }\vert\RunBorder_{w,u}(\delta_i,\delta_{i+1}-1)\vert\leq 3k
\}\mbox{.}\end{split}\]
We call the elements of $\StdPalFactrz_{w,k}(n_1,n_2)$ \emph{standard palindromic factorizations} of the factor $w[n_1,n_2]$. Note that $w[\delta_i,\delta_{i+1}-1]$ is either a standard palindrome or $w[\delta_i,\delta_{i+1}-1]$ has a bounded number of runs; in the latter case the factor $w[\delta_i,\delta_{i+1}-1]$ is not necessarily a palindrome. Also note that the border positions $\delta_i, \delta_{i+1}-1$ are not covered by runs; i.e. $\delta_i, \delta_{i+1}-1\in\rpoDom_{w,u}$. As such the images $\rpo_{w,u,\phi}(\delta_i), \rpo_{w,u,\phi(\delta_{i+1}-1)}$ are well defined.

We show that if there is a palindromic factorization, then there is also a standard palindromic factorization.

\begin{proposition}
\label{ue7dufyi89ie}
If $n_1\leq n_2\in\rpoDom_{w,u}$, $k=1+\maxPL(w[n_1,n_2])$ then $\StdPalFactrz_{w,u,k}(n_1,n_2)\not=\emptyset$.
\end{proposition}
\begin{proof}
Since $k=1+\maxPL(w[n_1,n_2])$, obviously $\PalFactrz_{w,k}(n_1,n_2)\not=\emptyset$.
Let $(m_1,m_2,\dots,m_j)\in\PalFactrz_{w,k}(n_1,n_2)$. We have that \begin{equation}\label{ry6yf68uy}w[n_1,n_2]=w[m_1,m_2-1]w[m_2, m_2-1]\cdots w[m_{j-1}, m_j-1]\mbox{.}\end{equation} 
Let \[\begin{split}\Omega=\{\maxCSP_{w,u}(m_i,m_{i+1}-1)\mid i\in\{1,2,\dots,j-1\}\mbox{ and } \\ \CnStdPal_{w,u}(m_i,m_{i+1}-1)\not=\emptyset\}\mbox{.}\end{split}\]
Let $i\in\{1,2,\dots,j-1\}$.
We distinguish following cases.
\begin{itemize}
\item If $\CnStdPal_{w,u}(m_i,m_{i+1}-1)=\emptyset$ then $(m_i,m_{i+1}-1)\in\widehat\Upsilon_{w,u}$.
\item 
If $\CnStdPal_{w,u}(m_i,m_{i+1}-1)\not=\emptyset$ then let $(\alpha_i,\beta_i)=\maxCSP_{w,u}(m_i,m_{i+1}-1)$. 
\begin{itemize}
\item If $\alpha_i=m_i$ then $w[m_i,m_{i+1}-1]=w[\alpha_i,\beta_i]$ and $(m_i,m_{i+1}-1)\in\StdPal_{w,u}$.
\item If $\alpha_i>m_i$ then \[w[m_i,m_{i+1}-1]=w[m_i,\alpha_i-1]w[\alpha_i,\beta_i]w[\beta_i+1,m_{i+1}-1]\mbox{,}\] 
$(m_i,\alpha_i-1), (\beta_i+1,m_{i+1}-1)\in\overline\Upsilon$, and $(\alpha_i,\beta_i)\in\StdPal_{w,u}$
\end{itemize}
\end{itemize}
From these cases and (\ref{ry6yf68uy}), it follows that there are $g\in\{1,2,\dots ,j-1\}$ and $\delta_1,\delta_2,\dots,\delta_g\in\{n_1, n_1+1, \dots, n_2\}$ 
  such that
\begin{itemize}
\item $n_1=\delta_1$, $n_2=\delta_g-1$, 
\item for every $(\alpha,\beta)\in\Omega$ there is $i\in\{1,2, \dots, g-1\}$ such that $(\alpha,\beta)=(\delta_i,\delta_{i+1}-1)$,
\item $w[n_1,n_2]=w[\delta_1, \delta_2-1]w[\delta_2,\delta_3-1]\cdots w[\delta_{g-1}, \delta_g-1]$, 
\item
if $i\in\{2,3,g-2\}$ and $(\delta_i,\delta_{i+1}-1)\not\in\Omega$ then \begin{equation}\label{ryg76rjt87}(\delta_{i-1}, \delta_i-1), (\delta_{i+1}, \delta_{i+2}-1)\in\Omega\mbox{, }\end{equation}
\item if $i\in\{1,2,3,g-1\}$ and $(\delta_i,\delta_{i+1}-1)\not\in\Omega$
then $(\delta_i,\delta_{i+1}-1)$ is a concatenation of at most $j\leq k$ words from $\Upsilon_{w,u}$.
\end{itemize} 
Note that if $\Omega=\emptyset$ then $g=2$ and $w[n_1,n_2]=[\delta_1,\delta_2-1]\not\in\Omega$.

Realize that $\Omega\subseteq\StdPal_{w,u}$.
From the definition of $\StdPal_{w,u}$ and (\ref{ryg76rjt87}) it follows that $\delta_i, \delta_{i+1}-1\in\rpoDom_{w,u}$ for every $i\in\{1,2,3,g-1\}$. Just recall that if $(i_1,i_2)\in\StdPal_{w,u}$ then $i_1-1,i_1,i_2,i_2+1\in\rpoDom_{w,u}$. 

It is easy to see that if $i_1\leq i_2\leq i_3\in\mathbb{N}_1$, $w[i_1,i_2]$, and $w[i_2+1,i_3]$ are such that $\vert\RunBorder_{w,u}(i_1,i_2)\vert\leq f_1\in\mathbb{N}_1$ and $\vert\RunBorder_{w,u}(i_2+1,i_3)\vert\leq f_2\in\mathbb{N}_1$ then $\vert\RunBorder_{w,u}(i_1,i_3)\vert\leq f_1+f_2+1$. It follows then from Lemma \ref{uj827xzv3} and Proposition \ref{ryy8hu8cxvxc} that if $(\delta_i,\delta_{i+1}-1)$ is a concatenation of at most $k$ words 
from $\Upsilon_{w,u}$ then $\vert\RunBorder_{w,u}(\delta_i,\delta_{i+1}-1)\vert\leq 3k$. 

Thus $(\delta_1,\delta_2,\dots,\delta_g)\in\StdPalFactrz_{w,u,k}(n_1,n_2)$.
This completes the proof.
\end{proof}

The next theorem presents an upper bound on the palindromic length of factors of the reduced words.
\begin{theorem}
\label{rug7r9unusedOK}
If $n_1\leq n_2\in\rpoDom_{w,u}$, $\overline n_1=\rpo_{w,u,\phi}(n_1)$, $\overline n_2=\rpo_{w,u,\phi}(n_2)$, and $k=1+\maxPL(w[n_1,n_2])$ then $\PL(\reduce_{w,u,\phi}[\overline n_1,\overline n_2])\leq 3k^3-3k^2$.
\end{theorem}
\begin{proof}
Proposition \ref{ue7dufyi89ie} implies that $\StdPalFactrz_{w,u,k}(n_1,n_2)\not=\emptyset$. 

Let $(\delta_1,\delta_2,\dots, \delta_g)\in\StdPalFactrz_{w,u,k}(n_1,n_2)$. We have that $g\leq k$. The definition of $\StdPalFactrz_{w,u,k}(n_1,n_2)$ asserts that for every $i\in\{1,2,\dots, g-1\}$ we have that $\delta_i, \delta_{i+1}-1\in\rpoDom_{w,u}$. 

Let $\overline \delta_i=\rpo_{w,u,\phi}(\delta_i)$, where $i\in \{1,2,\dots,g\}$.
It is easy to verify that if $j\in\mathbb{N}_1$ is such that $j-1,j\in \rpoDom_{w,u}$ then $\rpo_{w,u,\phi}(j-1)=\rpo_{w,u,\phi}(j)-1$.
It follows that if $i\in \{1,2,\dots,g\}$ then $\rpo_{w,u,\phi}(\delta_{i}-1)=\overline \delta_i-1$. In consequence we have that \[\overline w[\overline \delta_1, \overline \delta_g-1]=\overline w[\overline \delta_1, \overline \delta_2-1]\overline w[\overline \delta_2, \overline \delta_3-1]\cdots \overline w[\overline \delta_{g-1}, \overline \delta_g-1]\mbox{.}\]

The definition of $\StdPalFactrz_{w,u,k}$ says that  if $(\delta_i,\delta_{i+1}-1)\not \in\StdPal_{w,u}$ then \[\vert\RunBorder_{w,u}(\delta_i,\delta_{i+1}-1)\vert\leq 3k\mbox{.}\] Thus:
\begin{itemize}
\item
Proposition \ref{kkm3bvs56} implies that if $\vert\RunBorder_{w,u}(\delta_i,\delta_{i+1}-1)\vert\leq 3k$ and $\delta_i,\delta_{i+1}-1\in\rpoDom_{w,u}$ then \begin{equation}\begin{split}\label{ueyf7fujfy}\PL(\overline w[\overline \delta_i,\overline\delta_{i+1}-1])\leq (3k+1)\maxPL(w[\delta_1,\delta_g-1])=\\ (3k+1)(k-1)\leq 3k^2\mbox{.}\end{split}\end{equation}
\item
Proposition \ref{jt77gjri29k} implies that if $(\delta_i,\delta_{i+1}-1)\in\StdPal_{w,u}$ then $\PL(\overline w[\overline \delta_i,\overline\delta_{i+1}-1])=1$.
\end{itemize}

Then it follows from Lemma \ref{dujf720bxg} and (\ref{ueyf7fujfy}) that \[\begin{split}\PL(\overline w[\overline \delta_1, \overline \delta_g-1])\leq \sum_{i=1}^{g-1} \PL(\overline w[\overline \delta_i,\overline \delta_{i+1}-1])\leq \\ (g-1)(3k^2)\leq (k-1)(3k^2)= 3k^3-3k^2\mbox{.}\end{split}\]
This completes the proof.
\end{proof}

Now, we can step to the  proof of the main theorem of the current article.
\begin{proof}[Proof of Theorem \ref{d7uehgmn}]
Let $\gamma=3$. 

If $\PowFactor(u)\cap\Prefix(x)\cap\Alpha^{\gamma\vert u\vert}\not=\emptyset$, then let $t\in\Alpha^{+}$ and $w\in\Alpha^{\infty}$ be such that $x=tw$, $\Psi(x,u)\subseteq \Factor(w)$, $\maxPL(w)\leq \maxPL(x)$, and \[\PowFactor(u)\cap\Prefix(w)\cap\Alpha^{\gamma\vert u\vert}=\emptyset\mbox{.}\]  Since $x$ is non-ultimately periodic, such $t,w$ exist. (For example let $t\in\Prefix(x)$ be the shortest prefix such that $\PowFactor(u)\cap\Prefix(w)\cap\Alpha^{\gamma\vert u\vert}=\emptyset$.)  Then we have that $u\in\Pi_{\gamma}(w)$.

Let $h=3$, let $\phi\in\Phi_h$ be such that $\reduce_{w,u,\phi}$ is non-ultimately periodic and let $\overline w=\reduce_{w,u,\phi}$. Theorem \ref{fyur7t7tj} asserts that such $\phi$ exists.

From the definition of $\reduce_{w,u,\phi}$ it is clear that $\Psi(w,u)\subseteq\Factor(\overline w)$. 

Since $h=3$, Lemma \ref{tufdkd3455} implies that $u^5, (u^R)^5\not\in\Factor(\overline w)$. 

Let $n_1,n_2\in\mathbb{N}_1$.

If $\{n_1,n_1+2, \dots, n_2\}\cap\rpoDom_{\overline w,u}=\emptyset$ then obviously $\overline w[n_1,n_2]\in\Factor(z_i^{\phi(d_i)})$ for some $i\in\mathbb{N}_1$ and consequently $\PL(\overline w[n_1,n_2])\leq k$, because $z_i^{\phi(d_i)}\in\Factor(w)$.

If $\{n_1,n_1+2, \dots, n_2\}\cap\rpoDom_{\overline w,u}\not=\emptyset$ then 
let $\overline n_1=\min\{i\in\rpoDom_{\overline w,u}\mid i\geq n_1\}$, let $\overline n_2=\max\{i\in\rpoDom_{\overline w,u}\mid i\leq n\}$. Corollary \ref{dyf7ruryre} implies that there are  $m_1,m_2\in\rpoDom_{w,u}$ such that $\rpo_{w,u,\phi}(m_1)=\overline n_1$ and $\rpo_{w,u,\phi}(m_2)=\overline n_2$. 
Theorem \ref{rug7r9unusedOK} implies that $\PL(w[\overline n_1,\overline n_2])\leq 3k^3-3k^2$, since $\maxPL(w)\leq k$.

If $n_1<\overline n_1$ then obviously $\overline w[n_1,\overline n_1]\in\Factor(z_i^{\phi(d_i)})$ for some $i\in\mathbb{N}_1$. It follows that $\overline w[n_1,\overline n_1]\in\Factor(w)$ and consequently $\PL(\overline w[n_1,\overline n_1])\leq k$.
Analogously for $\overline n_2<n_2$. 
It follows from Lemma \ref{dujf720bxg} that \[\begin{split}\PL(\overline w[n_1,n_2])\leq  \PL(\overline w[n_1,\overline n_1-1])+\PL(\overline w[\overline n_1,\overline n_2])+\PL(\overline w[\overline n_2+1,n_1]) \\ \leq  2k+3k^3-3k^2\leq 3k^3\mbox{.}\end{split}\]

This completes the proof.
\end{proof}

\section*{Acknowledgments}
This work was supported by the Grant Agency of the Czech Technical University in Prague, grant No. SGS20/183/OHK4/3T/14.

\bibliographystyle{siam}
\IfFileExists{biblio.bib}{\bibliography{biblio}}{\bibliography{../!bibliography/biblio}}

\end{document}